\documentclass{lms}
\usepackage[all]{xy}
\usepackage{graphicx}
\usepackage{amssymb}
\usepackage{mathrsfs}
\usepackage{multirow}
\usepackage{float}
\usepackage{tikz-cd}
\usepackage{adjustbox}
\usepackage{amsmath}

\usepackage[bookmarksnumbered, bookmarksopen,
colorlinks,citecolor=blue,linkcolor=blue]{hyperref}

\newcounter{RomanNumber}
\newcommand{\MyRoman}[1]{\setcounter{RomanNumber}{#1}\Roman{RomanNumber}}

\begingroup
\catcode`\&=13
\gdef\pampmatrix{%
  \begingroup
  \let&=\amsamp
  \begin{pmatrix}%
}
\gdef\endpampmatrix{\end{pmatrix}\endgroup}
\endgroup
\newtheorem{theorem}{Theorem}[section] 
\newtheorem{lemma}[theorem]{Lemma}     
\newtheorem{corollary}[theorem]{Corollary}
\newtheorem{proposition}[theorem]{Proposition}

\newnumbered{assertion}{Assertion}    
\newnumbered{conjecture}{Conjecture}  
\newnumbered{definition}{Definition}
\newnumbered{hypothesis}{Hypothesis}
\newnumbered{remark}{Remark}
\newnumbered{note}{Note}
\newnumbered{observation}{Observation}
\newnumbered{problem}{Problem}
\newnumbered{question}{Question}
\newnumbered{algorithm}{Algorithm}
\newnumbered{example}{Example}
\newunnumbered{notation}{Notation} 

\title[Cancellation and homotopy rigidity]
 {Cancellation and homotopy rigidity of classical functors} 

\author{Ruizhi Huang and Jie Wu}

\dedication{}

\classno{55P65, 55P45 (primary), 55P35, 55P40, 55P30 (secondary)}

\extraline{The first named author is supported in part by Chinese Postdoctoral Science Foundation Grant No. 2018M631605.
The second named author is partially supported by the Singapore Ministry of Education research grant (AcRF Tier 1 WBS No. (R-146-000-222-112) and a grant (No. 11329101) of NSFC of China.}

\begin{document}
\maketitle

\begin{abstract}
We first show that simply connected co-$H$-spaces and connected $H$-spaces can be uniquely decomposed into prime factors in the homotopy category of pointed $p$-local spaces of finite type. This is used to develop a $p$-local version of Gray's correspondence between homotopy types of prime co-$H$-spaces and homotopy types of prime $H$-spaces, and the split fibration which connects them as well. Further, we use the unique decomposition theorem to study the homotopy rigidity problem for classic functors. Among others, we prove that $\Sigma \Omega$ and $\Omega$ are homotopy rigid on simply connected $p$-local co-$H$-spaces of finite type, and $\Omega\Sigma $ and $\Sigma$ are homotopy rigid on connected $p$-local $H$-spaces of finite type.  
\end{abstract}

\section{Introduction}
\noindent Cancellation and non-cancellation phenomena are widely investigated both in algebra and geometry. For instance, the Krull-Schmidt-Remak-Azumaya theorem \cite{Facchini} claims that any decomposition of an $R$-module into a direct sum of indecomposable modules is unique if the endomorphism rings of the summands are local rings. Comparing to algebraic cancellation, the corresponding problem in geometry is more mysterious which was illustrated by the classic example of Hilton and Roitberg \cite{Hilton}. In particular, they constructed an $H$-manifold $E_{7\omega}$ which is the total space of a $S^3$-bundle over $S^7$ classified by $7\omega$ with $\omega \in \pi_6(S^3)\cong \mathbb{Z}/12$ as the generator, and they proved that 
\[
Sp(2)\times S^3 \cong_{{\rm diff}} E_{7\omega}\times S^3, ~{\rm but}~ Sp(2) \not\simeq E_{7\omega}.
\]
However, we know that $Sp(2) \simeq_p E_{7\omega}$, i.e., they are locally homotopy equivalent at any prime $p$. This special but crucial example shed light on possible Krull-Schmidt type theorems in $p$-local homotopy theory, and indeed in 1975 Wilkerson \cite{Wilk} proved that each simply connected $p$-local finite $H$-space can be uniquely decomposed into a product of $H^\ast$-prime factors (also see Section $2$). The finite condition was essential there and was eliminated by Gray much later at the expense of considering $p$-complete spaces instead of $p$-local spaces. In that setting, Gray \cite{Gray} proved a Krull-Schmidt type theorem which states that each $p$-complete $H$-space can be uniquely decomposed to atomic pieces up to order and homotopy. Our first result in this paper is to drop the finite condition without other cost.
\begin{theorem}[(Theorem \ref{finitetype})]
$(1)$ Any simply connected $p$-local co-$H$-space of finite type can be uniquely decomposed into a wedge of irreducible factors up to order and homotopy equivalence;

$(2)$ Any connected $p$-local $H$-space of finite type can be uniquely decomposed into a weak product of irreducible factors up to order and homotopy equivalence.
\end{theorem}

In particular, the loop space of an irreducible co-$H$-space and the suspension of an irreducible $H$-space can be homotopically decomposed to irreducible pieces. As was pointed out by Gray \cite{Gray}, the least connected factors in the decompositions are of special interest, and this observation allows us to develop a $p$-local version of Gray's correspondence.
\begin{theorem}[(Theorem \ref{correspondence})]
There is $1$-$1$ correspondence in the sense of Gray between the homotopy types of connected $p$-local irreducible $H$-spaces $X$ of finite type and the homotopy types of simply connected $p$-local irreducible co-$H$-spaces $Y$ of finite type.
\end{theorem}

As in \cite{Gray}, we also call such a pair $(Y, X)$ a corresponding pair. Furthermore, there exists a fibration
$$X\stackrel{i}{\rightarrow} W\rightarrow Y$$
for some co-$H$-space $W$ and $i$ is null-homotopic. To state precisely, we have the following theorem.
\begin{theorem}[(Theorem \ref{pGray})]
Given a corresponding pair $(Y, X)$, there exists a homotopy equivalence
$$\Omega Y\simeq X\times \Omega W,$$
where $W$ is homotopy retract of $\Sigma(X\wedge X)$ and hence a co-$H$-space. Furthermore, $W\simeq \bigvee W_\alpha$ with each  $W_\alpha$ irreducible and

$1)$~$W_\alpha$ is homotopy retract of $[\Sigma\Omega Y]_n$ for some $n\geq 2$ (see Theorem \ref{ksuspension} for the definition of $[\Sigma\Omega Y]_n$);

$2)$~$\Sigma^{n-1}W_\alpha$ is a homotopy retract of $Y^{\wedge n}$;

$3)$~If $Y=\Sigma Z$, $W_\alpha$ is homotopy retract of $\Sigma Z^{\wedge n}$.
\end{theorem}

As an application of the unique decomposition theorem, the second part of the paper is devoted to the homotopy rigidity problem which was originally raised by Victor Buchstaber, and it was studied in \cite{Grbic} by Grbi\'{c} and the second author. Generally, it can be formulated in the following definition.
\begin{definition}[(\cite{Grbic}, Definiton $1$)]
Let $F:\mathcal{C}\rightarrow \mathcal{T}op$ be a homotopy functor from a subcategory of $\mathcal{T}op$. Then $F$ is called \textit{homotopy rigid} on $\mathcal{C}$ if for any $X$, $Y\in \mathcal{C}$,
$$F(X)\simeq F(Y) \iff X\simeq Y.$$
\end{definition}

It was also proved in \cite{Grbic} that $\Sigma \Omega$ and $\Omega$ are homotopy rigid on simply connected $p$-local finite co-$H$-spaces. One of our aims here is to drop the finite condition for the rigidity of $\Sigma\Omega$. Thanks to the unique decomposition theorem in the setting of finite type, we can not only generalize the result of \cite{Grbic} for co-$H$-spaces of finite type, but also can prove the rigidity property in the dual case.
\begin{theorem}[(Theorem \ref{coHfinitetype} and Theorem \ref{Hfinitetype})]\label{rigid}
$1)$~Let $Y_1$ and $Y_2$ be simply connected $p$-local co-$H$-spaces of finite type. If $\Sigma \Omega Y_1\simeq \Sigma \Omega Y_2$ (or $\Omega Y_1\simeq \Omega Y_2$), then $Y_1\simeq Y_2$;

$2)$~Let $X_1$ and $X_2$ be connected $p$-local $H$-spaces of finite type. If $\Omega\Sigma  X_1\simeq \Omega \Sigma X_2$ (or $\Sigma  X_1\simeq \Sigma  X_2$), then $X_1\simeq X_2$.
\end{theorem}

In particular, the second part of the above theorem confirms the Conjecture $1$ raised in \cite{Grbic}. Further, for the rigidity property on $H$-spaces, we can prove a more general result. Recall that in \cite{Selick} and \cite{Selick1}, Selick and the second author have showed that given any functorial coalgebra retract $A(V)$ of the tensor Hopf algebra $T(V)$, there exists a geometric realization $A(X)$ of $A(V)$ such that $A(X)$ is a functorial homotopy retract of $\Omega\Sigma X$ with the property that $E^{0}H_\ast(A(X)) \cong A(V)$ where $V=\tilde{H}_\ast(X)$. We may call such a homotopy retract a good one if $V$ is totally contained in $A(V)$ (see Definition \ref{goodretract} for precise definition). Under this condition, we can prove the following:
\begin{theorem}[(Theorem \ref{A})]
Let $X_1$ and $X_2$ be connected $p$-local $H$-spaces of finite type. Then if $A(X)\simeq A(Y)$, we have $X\simeq Y$.
\end{theorem}

A natural question is whether the dual of the above theorem is true or not, which may be also served as a potential generalization of the part $1)$ of Theorem \ref{rigid}. To make the question more explicit, we have to use a more general functorial decomposition of loops on co-$H$-spaces introduced in \cite{Selick2}. For the earlier good functorial coalgebra retract $A(V)$ of $T(V)$, there exists a geometric realization $\bar{A}(Y)$ which is a functorial homotopy retract of $\Omega Y$ for any simply connected $p$-local co-$H$-space $Y$ of finite type. Then we may formulate the following natural question:

\begin{question}
Is the functor $\bar{A}$ homotopy rigid on simply connected $p$-local co-$H$-spaces of finite type?
\end{question}

There is also a parallel question of rigidity for the integral case, which should be much more difficult due to the existence of Mislin genus. For a given nilpotent space $X$, the \textit{Mislin genus} $G(X)$ is defined to be the set of all the homotopy types of nilpotent spaces $Y$ such that $Y\simeq_p X$ at every prime $p$. By the work of Hilton and Roitberg, we know that $G(Sp(2))=\{Sp(2), E_{7\omega}\}$. In \cite{Grbic}, Grbi\'{c} and the second author proved the integral rigidity of $\Sigma\Omega$ for finite co-$H$-spaces with finite homology by using a result of McGibbon concerning Mislin genus. However, since Mislin genus in general is rather mysterious, it should be very hard to prove integral rigidity by combining the local results. It will be interesting if one can find other ways to detect the rigidity problem in the integral case.

We also discuss some other classic functors in the appendix including the free loop functor $L$, where by definition $L(X)={\rm map}(S^1, X)$. There is a canonical fibration
$$\Omega X \rightarrow L(X) \rightarrow X,$$
from which we see $L(X)$ and $\Omega X$ are closely related. Hence, it is reasonable to expect some rigidity results for $L$. 
\begin{question}
Can we choose a suitable and meaningful category and condition to prove the rigidity of $L$?
\end{question}
Besides above, there is another interesting question based on the observation that both the results and their proofs in this paper are rather dual to each other. Further, the functors $\Omega$ and $\Sigma$ are adjoint to each other and the operations $\times$ and $\vee$ refer to the product and co-product respectively in categorical sense. Of course, the operation $\wedge$ will play a special and important role and exists in one of the two categories here. Hence, it may be reasonable to consider the homotopy rigidity problem in a pure categorical setting. If this is possible, it may be expectable to apply this categorical analogue to other families of adjoint functors. At least, one can exploit the relationship of rigidity properties of adjoint functors under some further assumptions. Let us illustrate this point in more detials. Start with some model category $(\mathcal{U}, \simeq_w)$ and its full subcategories $\mathcal{C}$ and $\mathcal{D}$ which are closed under taking countable limit and colimit respectively. Suppose $\mathcal{D}=(\mathcal{D},\otimes )$ is a distributive monoidal category and we have a pair of adjoint homotopy functors
\[
   \begin{tikzcd}
   \mathcal{C}\ar[r, shift left ]{r}{F} & \mathcal{D}\ar[l, shift left ]{l}{G}.
   \end{tikzcd}
\]
Further, $F$ sends any morphism in $\mathcal{C}$ to cofibration, and $G$ sends any morphism in $\mathcal{D}$ to fibration. 
\begin{question}
$1)$ Under some suitable conditions, does the homotopy rigidity property of $F$ determine that of $G$, or vice versa?

$2)$ Can we make some reasonable assumptions such that $F$ and $G$ are homotopy rigid on their respective categories?
\end{question}

The paper is organized as follows. In section $2$, we will prove the unique decomposition theorem for finite type case. As a direct application, we will develop the $p$-local version of Gray's correspondence in section $3$. Section $4$ and $5$ are devoted to prove the rigidity property of $\Sigma\Omega$ and $A$ respectively. At last section, we will also discuss the rigidity problem for some other canonical homotopy functors including the free loop functor as an appendix. We also remark that throughout this paper all spaces under consideration are supposed to be connected and $p$-local, and all (co)homology groups will have $\mathbb{Z}/p$ coefficients unless otherwise stated.

\section{Cancellation and unique decomposition theorem}
\noindent Cancellation problems are usually solved by the proofs of unique decompositions. The philosophy of unique decomposition has been largely applied to many mathematical subjects. For instance, Dedekind domain is known to be the appropriate concept to study the ring of algebraic integers in number fields, the ideals of which are uniquely decomposed into prime ideals. The key point is what is the suitable definition for `primes'. For homotopy theory, Wilkerson in \cite{Wilk} proposes the following concept as a candidate:

\begin{definition}\label{Hprime}
Let $X$ be a connected nilpotent $p$-local space of finite type, we call $X$ \textit{an $H^\ast$-prime space} if for every self map $f: X\rightarrow X$ either: 

$(1)$ $H^\ast(f): H^\ast(X; \mathbb{Z}/p)\rightarrow H^\ast(X; \mathbb{Z}/p)$ is an isomorphism, or 

$(2)$ $H^\ast(f)$ is \textit{weakly nilpotent}, i. e., for each $n>0$, there exists integer $N_n$ such that $(H^\ast(f))^{N_n}(H^m(X; \mathbb{Z}/p))=0$ for $0<m\leq n$.

We notice that (1) here is equivalent to

$(1)^\prime$ $f$ is a homotopy equivalence.
\end{definition}

To perform decomposition, we also need an operation as an analogy to product for numbers which depends on the choice of category we work with.
We consider two classic settings: the category of $H$-spaces with Cartesian product $\times$ as the product and the category of co-$H$-spaces with wedge product $\vee$. The following property of $H^\ast$-prime space justifies its name:

\begin{lemma}[(\cite{Wilk}, Lemma $1.6$)]\label{prime}
Let $X$, $Y$ and $W$ be connected nilpotent $p$-local CW-complexes of finite type. Suppose that $X$ is $H^\ast$-prime, we have

$(1)$ if $X$ is retract of $Y\times W$, then $X$ is either a retract of $Y$ or of $W$ by the canonical compositions,

$(2)$ if $X$ is retract of $Y\vee W$, then $X$ is either a retract of $Y$ or of $W$ by the canonical compositions.
\end{lemma}

It is exactly this property which is essential for Wilkerson's cancellation, and we may make the following definition for future use:
\begin{definition}\label{prime2}
$(1)$ Let $X$ be a connected $p$-local $H$-space of finite type, we call $X$ \textit{a prime $H$-space} if whenever $X$ is a homotopy retract of $Y\times W$ where $Y$ and $W$ are connected $p$-local $H$-spaces of finite type, $X$ will be a homotopy retract of $Y$ or $W$.

$(2)$ Let $X$ be a simply connected $p$-local co-$H$-space of finite type, we call $X$ \textit{a prime co-$H$-space} if whenever $X$ is a homotopy retract of $Y\vee W$ where $Y$ and $W$ are simply connected $p$-local co-$H$-spaces of finite type, $X$ will be a homotopy retract of $Y$ or $W$.
\end{definition}

On the other hand, a space X is called \textit{irreducible} if X has no nontrivial homotopy retracts. It is easy to see that an $H$-space (or a co-$H$-space) $X$ is irreducible if and only if it is homotopically indecomposable. The main theorem of \cite{Wilk} provided positive answers for cancellation under some finite conditions:

\begin{theorem}[(\cite{Wilk})]\label{finite}
Let $X$ be a connected nilpotent $p$-local space of finite type, we have

$(1)$ if $X$ is an $H$-space which is finite or only has finitely many nontrivial homotopy groups, then $X\simeq X_1\times X_2\times \cdots
\times X_m$ with each $X_i$ irreducible as $H$-space and the decomposition is unique up to order.

Furthermore, $X$ is an irreducible as $H$-space if and only if $X$ is $H^\ast$-prime,

$(2)$ if $X$ is a finite co-$H$-space, then $X\simeq X_1\vee X_2\vee \cdots
\vee X_n$ with each $X_i$ irreducible as co-$H$-space and the decomposition is unique up to order.

Furthermore, $X$ is an irreducible as co-$H$-space if and only if $X$ is $H^\ast$-prime.
\end{theorem}

We notice that Wilkerson's unique decomposition theorem only holds for homotopically or  homologically finite spaces. Our goal in this section is to generalize his theorem to the finite type case, before which we need some preparation. Given any simply connected space $X$ of finite type, there is a minimal cell complex $\tilde{X}$ such that $\tilde{X}\simeq X$ \cite{Hatcher}. For any such minimal model $\tilde{X}$ of $X$, we may define for any $n$
\[
\tilde{{\rm sk}}^{\tilde{X}}_n(X) ={\rm sk}_n (\tilde{X}).
\]
Then in the second part of the following lemma, we will see that $\tilde{{\rm sk}}^{\tilde{X}}_n(X)$ is independent of the choice of the minimal model, and the homotopy type of $X$ as well. Hence we have a sequence of homotopy functors $\tilde{{\rm sk}}_n$'s defined by $\tilde{{\rm sk}}_n(X)\simeq\tilde{{\rm sk}}^{\tilde{X}}_n(X)$  for any $X$ and $n$, where $\tilde{X}$ is any minimal model of $X$.

\begin{lemma}\label{piece}
Let $X$ and $Y$ be two connected nilpotent $p$-local spaces of finite type, if $X\simeq Y$ we have

$(1)$ $P^m(X)\simeq P^m(Y)$ for any $m$ where $P^m(X)$ and $P^m(Y)$ are the $m$-th stages of the Postnikov systems of $X$ and $Y$ respectively.

$(2)$ $\tilde{{\rm sk}}_n$ is well defined for each $n$ and $\tilde{{\rm sk}}_n(X)\simeq \tilde{{\rm sk}}_n(Y)$ for any $n$ provided further $X$ is simply connected.
\end{lemma}

\begin{proof}
$(1)$ immediately follows from the construction of functorial Postnikov system as indicated in $\MyRoman{9}. 2$ of \cite{White}, for instance.

$(2)$ Since $X$ and $Y$ are simply connected, we may first suppose $X$ and $Y$ are minimal cell complexes and consider the usual skeleton. Choose a homotopy equivalence $f: X\rightarrow Y$, then we have the restriction map $f: {\rm sk}_n(X)\rightarrow {\rm sk}_n(Y)$ and the following commutative diagram of homology:
\[
  \begin{tikzcd}
    \tilde{H}_n({\rm sk}_n(X); \mathbb{Z}_{(p)})\ar[two heads]{r}{i_\ast} \ar{d}{f_\ast} & \tilde{H}_n(X; \mathbb{Z}_{(p)})\ar{d}{f_\ast}[swap]{\cong}\\
    \tilde{H}_n({\rm sk}_n(Y); \mathbb{Z}_{(p)}) \ar[two heads]{r}{i_\ast} & \tilde{H}_n(Y; \mathbb{Z}_{(p)}).
  \end{tikzcd}
\]
After applying the fundamental decomposition for finitely generated abelian groups the diagram becomes
\[
   \begin{tikzcd}
   \underbrace{\mathbb{Z}_{(p)}\oplus \mathbb{Z}_{(p)}\cdots \oplus \mathbb{Z}_{(p)}}_{m+s} \ar[two heads]{r}{\pi} \ar{d}{\tilde{q}} &
    \underbrace{\mathbb{Z}_{(p)}\oplus \mathbb{Z}_{(p)}\cdots \oplus \mathbb{Z}_{(p)}}_{m}\oplus\underbrace{\mathbb{Z}/p^{r_1}\oplus \mathbb{Z}/p^{r_2}\cdots \oplus \mathbb{Z}/p^{r_s}}_{s}\ar{d}{q}[swap]{\cong}\\
    \underbrace{\mathbb{Z}_{(p)}\oplus \mathbb{Z}_{(p)}\cdots \oplus\mathbb{Z}_{(p)}}_{m+s} \ar[two heads]{r}{\pi}  &
    \underbrace{\mathbb{Z}_{(p)}\oplus \mathbb{Z}_{(p)}\cdots \oplus \mathbb{Z}_{(p)}}_{m}\oplus\underbrace{\mathbb{Z}/p^{r_1}\oplus \mathbb{Z}/p^{r_2}\cdots \oplus \mathbb{Z}/p^{r_s}}_{s}.
   \end{tikzcd}
\]
Then $\tilde{q}$ can be expressed as a matrix lifted from that of $q$ after fixing a basis. Since $q$ is an isomorphism and a scaler multiplication by a number prime to $p$ is invertible in the $p$-local ring, then the matrix is invertible and of the form
\[
  \begin{pmatrix}
  A & 0\\
  B & C
  \end{pmatrix}
\]
where the block partition corresponds to that of torsion free part and torsion part. Hence $\tilde{q}$ is an isomorphism and then $f_\ast$ is an isomorphism.  Since $f$ clearly induces isomorphisms of homology in other dimensions, we see that $f$ is a homotopy equivalence on the $n$-th skeleton.

Return to the lemma, if $\tilde{X}_1$ and $\tilde{X}_2$ are two minimal models of $X$, then $\tilde{X}_1\simeq \tilde{X}_2$. Hence according to the previous discussion, we have
\[
\tilde{{\rm sk}}^{\tilde{X}_1}_n(X) ={\rm sk}_n (\tilde{X}_1)\simeq {\rm sk}_n (\tilde{X}_2)=\tilde{{\rm sk}}^{\tilde{X}_2}_n(X),
\]
which means $\tilde{{\rm sk}}^{\tilde{X}_1}_n(X)$ is independent of the choice of the minimal model. Further, it is easy to see two homotopy equivalent spaces have the same set of minimal models. Hence $\tilde{{\rm sk}}_n$ is a homotopy functor, and the lemma follows.
\end{proof}

\begin{lemma}\label{prime=irred}
$(1)$ A connected $p$-local $H$-space of finite type is prime if and only if it is irreducible.

$(2)$ A simply connected $p$-local co-$H$-space of finite type is prime if and only if it is irreducible.
\end{lemma}
\begin{proof}
$(1)$ A prime $H$-space is clearly irreducible. Conversely, given any irreducible $p$-local $H$-space $X$ of finite type, we want to prove that $X$ is prime. Suppose $X$ is a homotopy retract of $Y\times W$, then in the Postnikov system the $m$-th stage $P^m(X)$ is a homotopy retract of $P^m(Y)\times P^{m}(W)\simeq P^m(Y\times W)$ for any $m$.

In order to prove that $X$ is a homotopy retract of $Y$ or $W$, we need to analyze the Postnikov system carefully, and one of the technical difficulties is due to the existence of the phantom phenomena. In particular, Harper and Roitberg \cite{Harper92} and Gray \cite{Gray66} used the theory of phantom maps to construct some spaces $K_1$ and $K_2$ such that $P^m(K_1)\simeq P^m(K_2)$ for any $m$ but $K_1\not\simeq K_2$.

To start, we may view the Postnikov system of $X$ as an infinite forest (Figure $1$) by applying Wilkerson's unique decomposition (Theorem \ref{finite}) for each stage of the Postnikov system.

\begin{figure}[H]
\centering
\includegraphics[width=3.2in]{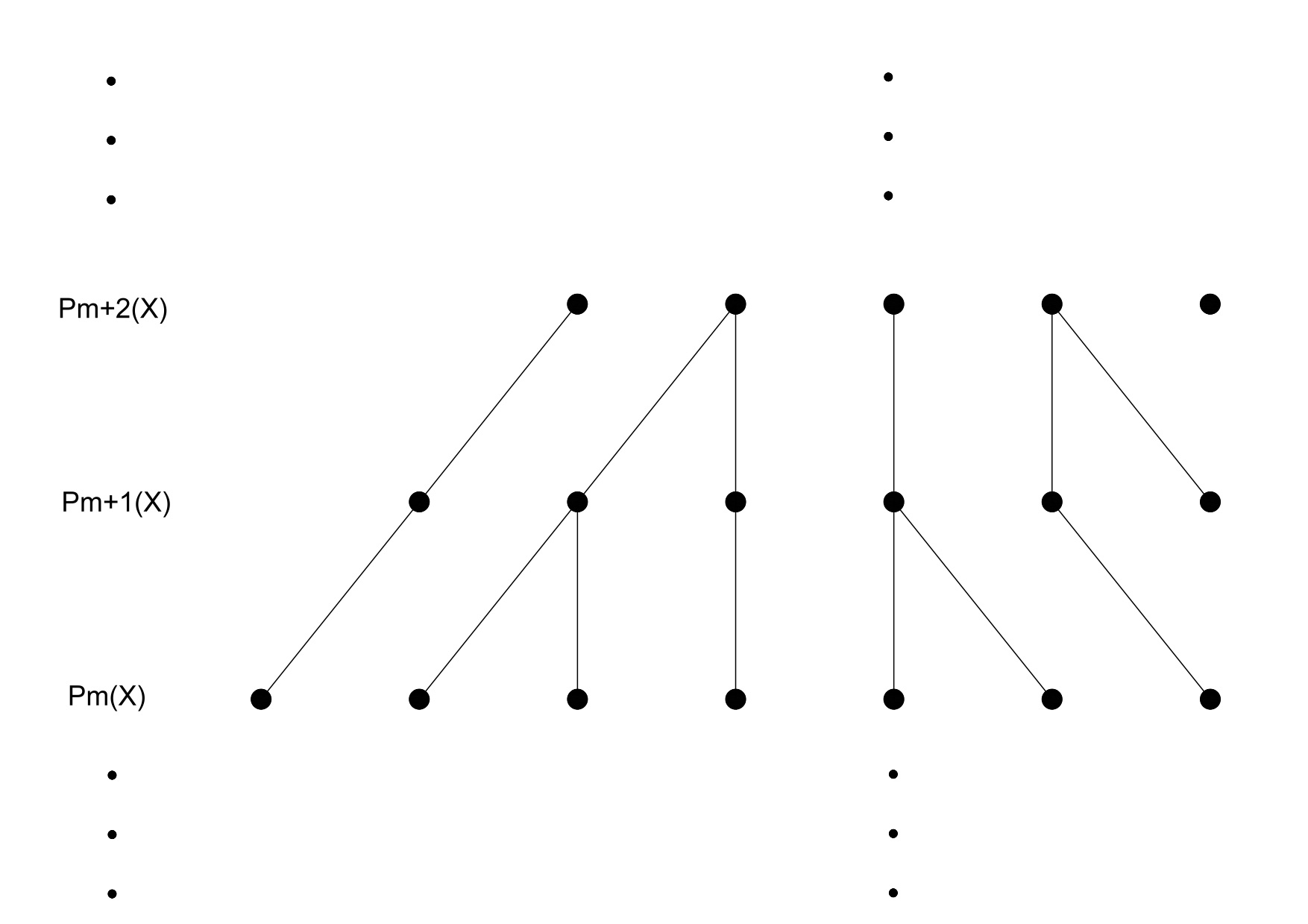}
\caption{Forest of Postnikov system of an $H$-space}\label{equi}
\end{figure}

In the forest, each node represents an irreducible component of the corresponding stage, and two nodes are connected iff the one in the lower stage is a factor of the corresponding stage of the Postnikov system of the other one. Then we should notice that

1) there are only finitely many nodes in each stage by Wilkerson's unique decomposition;

2) any two nodes in the same stage can not be connected to a common node in the lower stage;

3) Each of the nodes in some stage should be connected to some node in the one higher stage;

4) it is possible that at every stage there may be some `half-isolated nodes' that they are not connected to any nodes in the lower stage. Hence the number of nodes in each stage may not decrease with respect to the stage. 

For instance, let 
\[X=E\times \prod_{l\geq n+1}^{\infty}K(\mathbb{Z}/p, l)\] 
such that $E$ is the homotopy fibre of a map 
\[k=k_1\times k_2: K(\mathbb{Z}/p\oplus \mathbb{Z}/p, m)\simeq K(\mathbb{Z}/p, m)\times K(\mathbb{Z}/p, m) \rightarrow K(\mathbb{Z}/p, n+1)\]
with $n>m>1$ and further both $k_1$ and $k_2$ are homotopically nontrivial; in other word, $E$ is irreducible and admits a $2$-stage Postnikov system classified by the $k$-invariant $k$. 
Then the associated forest of $X$ is as showed in Figure $2$, where the nodes $1$ and $2$ represent the two Eilenberg-MacLane spaces  $K(\mathbb{Z}/p, m)$, the node $3$ represents $E\simeq P^n(E)$ and the half-isolated node $4$ represents $K(\mathbb{Z}/p, n+1)$ and so on (note that we used a reducible space as the example for purposes of simplicity).

\begin{figure}[H]
\centering
\includegraphics[width=3.2in]{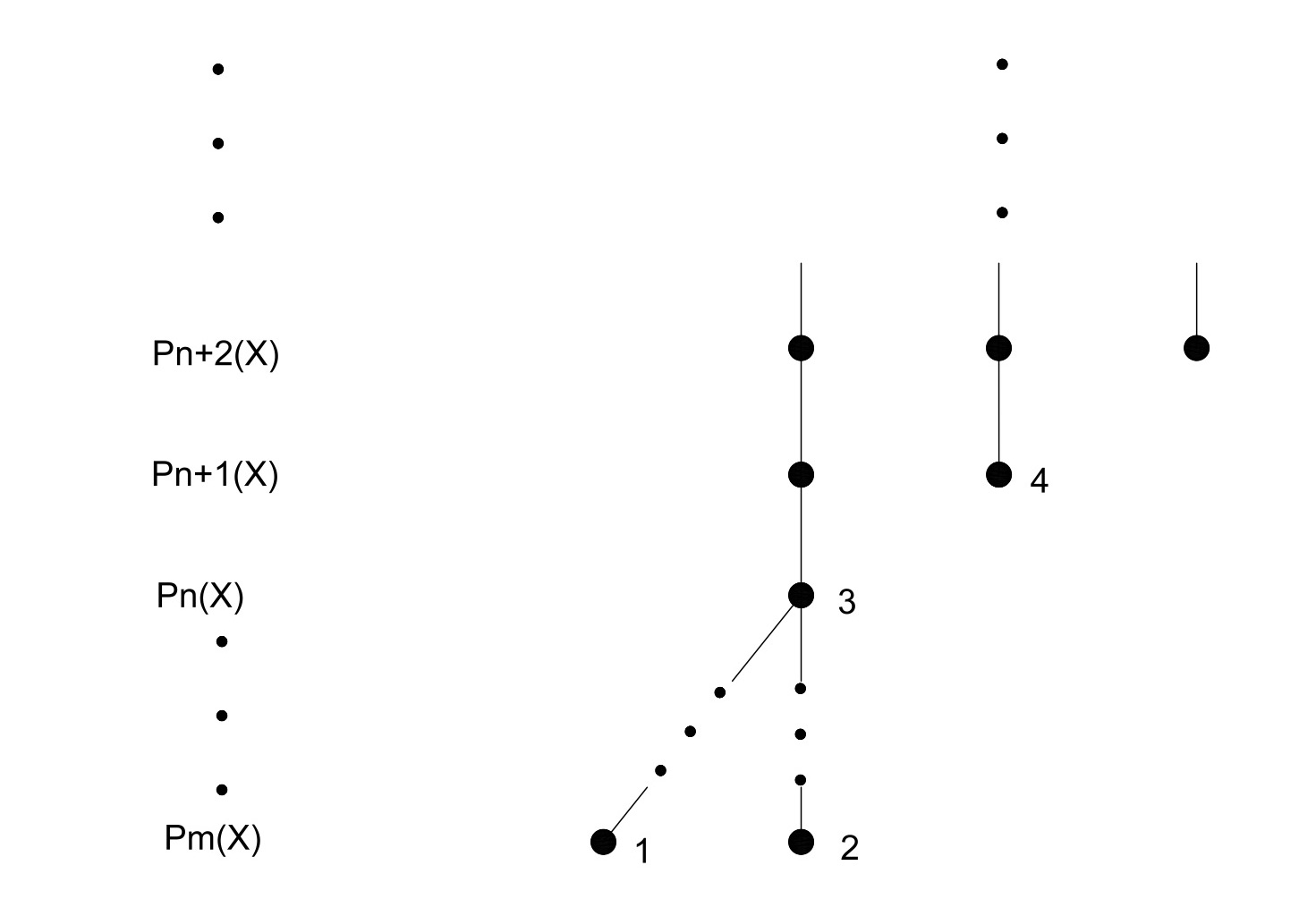}
\caption{An example of the forest associated to the Postnikov system}
\end{figure}

Now for any positive integers $m$ and $n$ such that $m>n$, we may define a number
\[
N_{m,n}:=\sharp \big\{ Z~|~P^m(X)\simeq \cdots\times Z\times \cdots ~{\rm with}~Z~{\rm irreducible}~; P^k(Z)\not\simeq \ast, \exists~k<n
\big\},
\]
where the multiplicity of $Z$ should be taken into account. With the help of Figure $3$, we see that $N_{m,n}$ is the number of $Z$'s such that $Z$ is rooted in some $k$-stage with $k<n$, i.e. is connected to $k$-stage by some path. Then $N_{m+1,n}\leq N_{m,n}$  for any $m>n$ by $2)$ and $3)$. Hence the sequence $\{N_{m,n} \}_m$ for any fixed $n$ should be stable eventually. Since $X$ is irreducible, the stable value should be $1$, i.e. $N_{m,n}=1$ for sufficiently large $m$ (\textbf{Note:} This does not mean the corresponding stage is irreducible). Choose any such $m$, we have a decomposition
$$P^{m}(X)\simeq  Z\times X^\prime,$$
where $Z$ is the only irreducible component which contributes to $N_{m,n}$. An clear but important fact is  $|X^\prime|\geq n-1$ (here $|~|$ refers to the connectivity).
Then by our earlier assumption, $Z$ is a homotopy retract of $P^{m}(Y)\times  P^{m}(W)$, and then of $P^m(Y)$ or $P^{m}(W)$ by Lemma \ref{prime}.
 \begin{figure}[H]
\centering
\includegraphics[width=2.9in]{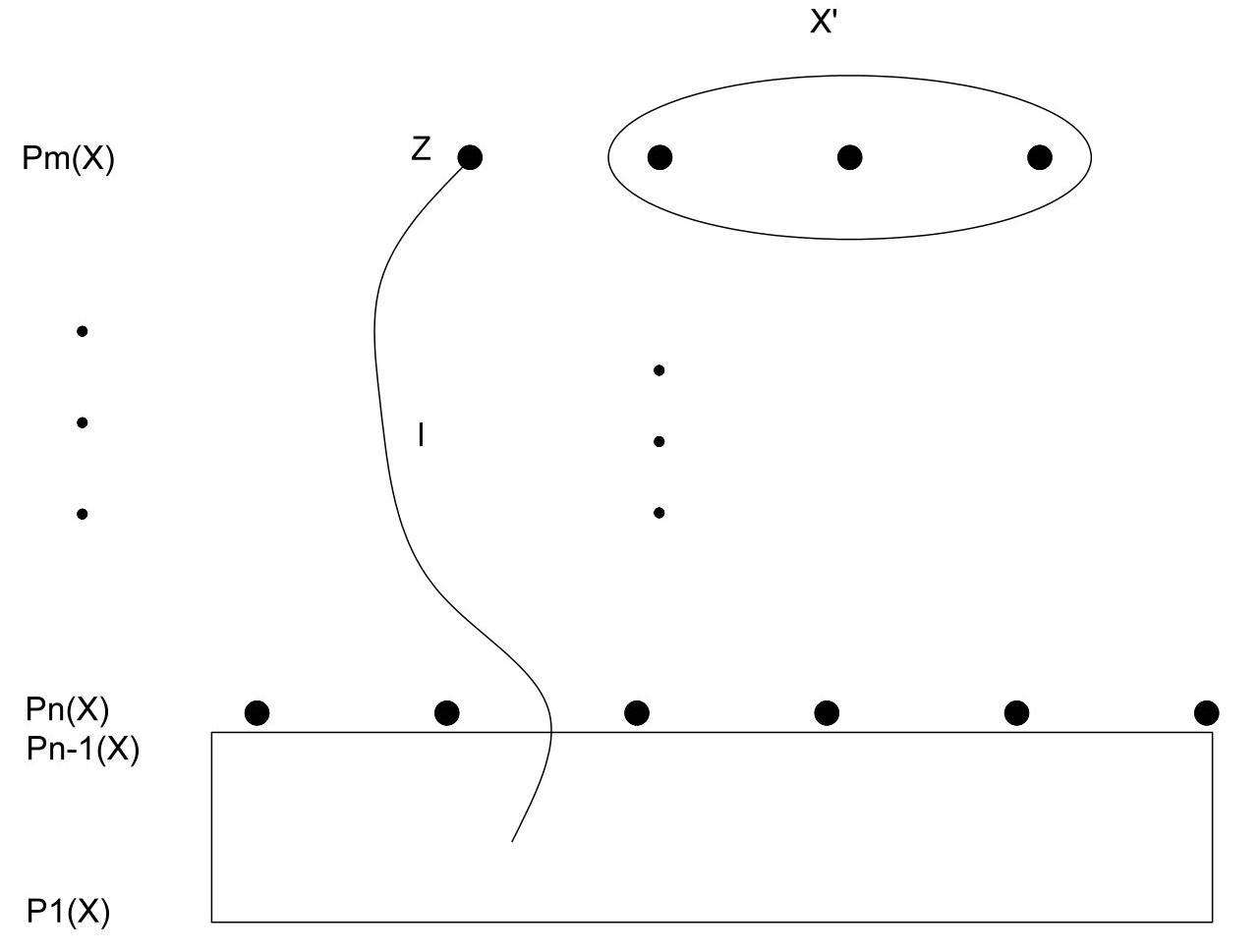}
\caption{To find irreducible factor rooted in lower stages}\label{equi}
\end{figure}

Now we want to iterate this procedure. Start with any $n=n_0$. By the above argument there exists $n_1>n_0$ such that
$$P^{n_1}(X)\simeq  Z_1\times X_1^\prime,$$
$Z_1$ is irreducible which is a homotopy retract of $P^{n_1}(Y)$ or $P^{n_1}(W)$ and $|X_1^\prime|\geq n_0-1$.

For $n_1$, there exists $n_2>n_1$ such that
$$P^{n_2}(X)\simeq  Z_2\times X_2^\prime,$$
$Z_2$ is irreducible and a homotopy retract of $P^{n_2}(Y)$ or $P^{n_2}(W)$ and $|X_2^\prime|\geq n_1-1$. And by our choice, we notice that $Z_1$ is one of the components of $P^{n_1}(Z_2)$, and $Z_2$ can not be a homotopy retract of $P^{n_2}(Y)$ if $Z_1$ is not a retract of $P^{n_1}(Y)$.

Iterating the above process, we get a strictly increasing sequence of numbers $n_k>n_{k-1}>\ldots >n_0$ such that for each $1\leq i\leq k$,
$$P^{n_i}(X)\simeq  Z_i\times X_i^\prime,$$
and w.l.o.g., the irreducible $H$-space $Z_i$ is a homotopy retract of $P^{n_i}(Y)$ and also of $P^{n_i}(Z_{i+1})$, and $|X_i^\prime|\geq n_{i-1}-1$. Hence $\pi_\ast (Z_i)\cong \pi_\ast(X)$ for $\ast\leq n_{i-1}-1$ which implies ${\rm holim}_{k}Z_k\simeq X$ (hence, the spaces $X_i^\prime$ can be viewed as the redundant parts of the Postnikov system and we only need to trace the spaces $Z_i$).

\begin{figure}[H]
\centering
\includegraphics[width=2.9in]{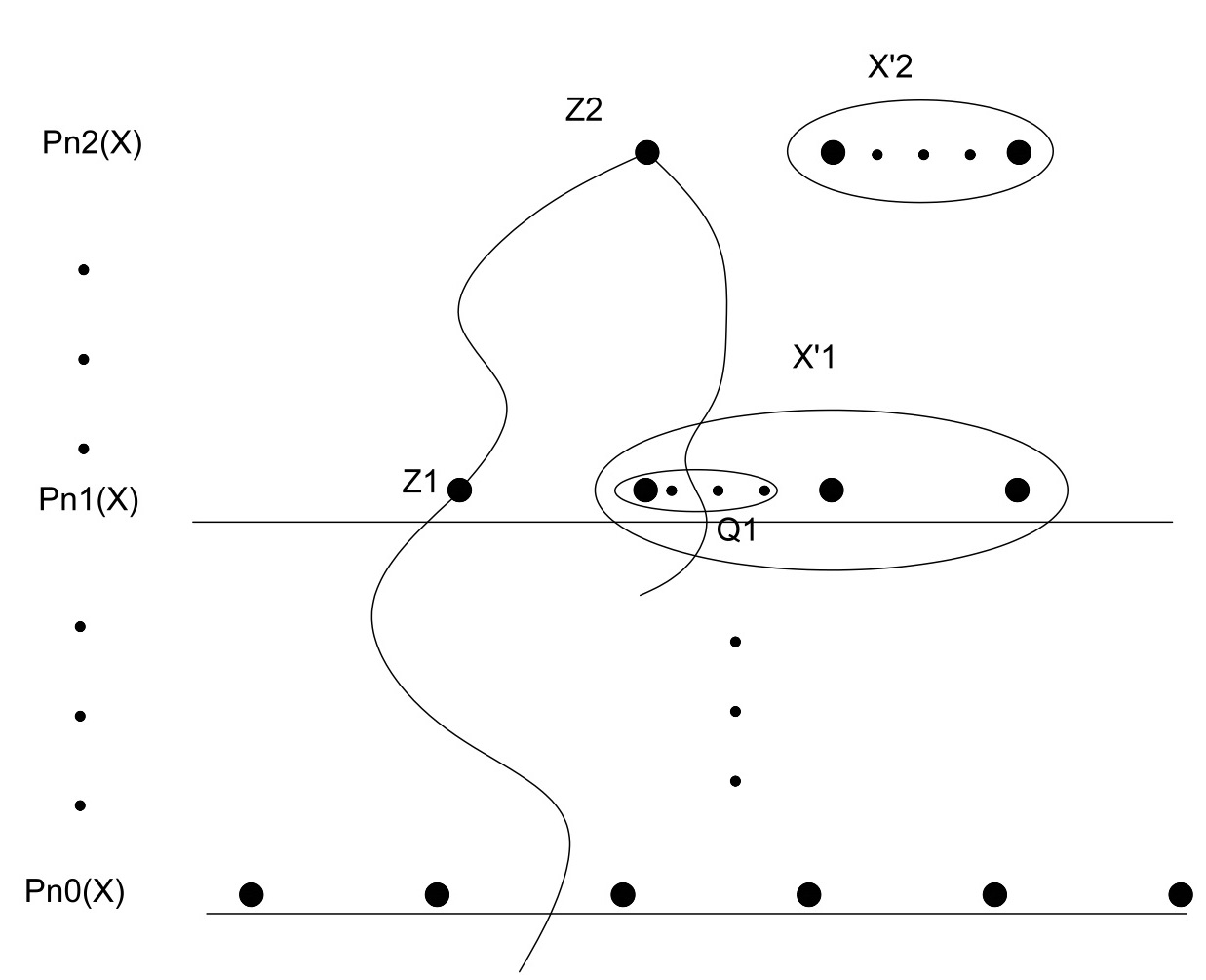}
\caption{First two steps of the iterating process}\label{equi}
\end{figure}

Furthermore, there exist decompositions
$$P^{n_i}(Z_{i+1})\simeq Z_i \times Q_i,$$
$$P^{n_{i+1}}(Y)\simeq Z_{i+1} \times Y_{i+1},$$
where $Q_i$ is a homotopy retract of $X_i^\prime$, and then $|Q_i|\geq|X_i^\prime|\geq n_{i-1}-1$.
Hence we have
$$P^{n_i}(Y)\simeq Z_i \times Q_i\times P^{n_i}(Y_{i+1}),$$
and we may define $$Y_i^\prime \simeq Z_i\times P^{n_i}(Y_{i+1}).$$
Then we see ${\rm holim}_{k} Y_i^\prime \simeq Y$, and also $Z_i$ is a homotopy retract of $Y_i^\prime$.
Passing to the limit, we see that $X$ should be a homotopy retract of $Y$.

$(2)$~The proof is exactly dual to that of $(1)$.
\end{proof}

Now we are ready to prove the unique decomposition theorem for the finite type case.
\begin{theorem}\label{finitetype}
Let $X$ be a connected nilpotent $p$-local space of finite type, we have

$(1)$ if $X$ is a $H$-space, then
\begin{equation}\label{decomH}
X\simeq \times_iX_i
\end{equation}
with each $X_i$ irreducible as $H$-space and the decomposition is unique up to order.

$(2)$ if $X$ is a simply connected co-$H$-space, then
\begin{equation}\label{decomcoH}
X\simeq \vee_iX_i
\end{equation}
with each $X_i$ irreducible as co-$H$-space and the decomposition is unique up to order.
\end{theorem}

\begin{proof}
The existence of such decompositions into irreducible components follows by inductions on the connectivity and the finite type assumption. For uniqueness, since irreducible is equivalent to prime by Lemma \ref{prime=irred} in both cases, a similar argument as that in the proof of Theorem $1.7$ in \cite{Wilk} will provide the proof.
\end{proof}

\section{$p$-local version of Gray's correspondence}
\noindent In \cite{Gray}, Gray described a $1$-$1$ correspondence between atomic $H$-spaces and atomic co-$H$-spaces in the $p$-complete category based on the unique decomposition theorem in that setting. Since we have already proved the unique decomposition theorem in the $p$-local setting (Theorem \ref{finitetype}), we may develop a parallel theory. Notice in this situation, we have to use irreducible or equivalently prime spaces instead of atomic ones. We first give the following definition as in \cite{Gray}:
\begin{definition}\label{pair}
We call a pair of connected $p$-local prime spaces $(Y, X)$ of finite type \textit{a corresponding pair} if there are structure maps $f$, $g$, $g^\prime$ and $h$ such that the compositions
$$X\stackrel{g}{\rightarrow} \Omega Y \stackrel{h}{\rightarrow} X,$$
$$Y\stackrel{f}{\rightarrow} \Sigma X \stackrel{g^\prime}{\rightarrow}Y$$
are homotopic to identity.
\end{definition}

To develop the corresponding theory, we have to recall two decomposition theorems due to Grbi{\'c}, Theriault and the second author.
\begin{theorem}[(\cite{Grbic2}, Theorem $1.2$)]\label{suspension}
Let $A(V)$ be any functorial coalgebra retract of $T(V)$. Then there exists a geometric realization functor $\bar{A}(Y)$ of $A(V)$ which is also a functorial homotopy retract of $\Omega Y$, where $V=\Sigma^{-1}\tilde{H}_\ast(Y)$ and $Y$ is any $p$-local simply connected co-$H$-space of finite type. Furthermore, for any $p$-local connected co-$H$-space $Z$, there is a functorial homotopy decomposition
\[
Z\wedge \bar{A}(Y)\simeq \bigvee_{n=1}^{\infty}[Z\wedge \bar{A}(Y)]_n,
\]
where $[Z\wedge \bar{A}(Y)]_n$ is a space with the property that
$$\tilde{H}_\ast ([Z\wedge \bar{A}(Y)]_n)\cong \tilde{H}_\ast (Z)\otimes A_n(\Sigma^{-1}\tilde{H}_\ast(Y)).$$
\end{theorem}

\begin{theorem}[(Generalized Hilton-Milnor Theorem, \cite{Grbic3}, Theorem $1.2$)]\label{Hilton}
Let $Y_1, \ldots, Y_m$ be simply-connected co-$H$-spaces. There exists a natural homotopy decomposition
\[
\Omega(Y_1\vee \cdots \vee Y_m)\simeq \prod_{\alpha\in \mathcal{I}}\Omega M((Y_i,\alpha_i)_{i=1}^m),
\]
where $\mathcal{I}$ runs over a vector space basis of the free Lie algebra $L\langle y_1,\ldots, y_m\rangle$ and each $M((Y_i,\alpha_i)_{i=1}^m)$ is a simply-connected co-$H$-space such that
$$\tilde{H}_\ast(M((Y_i,\alpha_i)_{i=1}^m))\cong \Sigma \Big((\Sigma^{-1}\tilde{H}_\ast(Y_1))^{\otimes \alpha_1}\otimes \cdots \otimes (\Sigma^{-1}\tilde{H}_\ast(Y_m))^{\otimes \alpha_m}
\Big).$$
\end{theorem}

Now we are ready to prove the correspondence theorem.

\begin{theorem}\label{correspondence}
There is $1$-$1$ correspondence in the sense of Definition \ref{pair} between the homotopy types of connected $p$-local irreducible $H$-spaces of finite type and the homotopy types of simply connected $p$-local irreducible co-$H$-spaces of finite type.
\end{theorem}
\begin{proof}
On the one hand, given any connected $p$-local irreducible $H$-space $X$ of finite type, we have a decomposition
$$\Sigma X\simeq Y\vee Q,$$
by Theorem \ref{finitetype}, where $Y$ is any irreducible factor such that $|Y|=|\Sigma X|$ (here we use $|\Sigma X|$ to denote the connectivity of $\Sigma X$). Then $Y$ is a simply connected co-$H$-space. Again by Theorem \ref{finitetype} there exists a decomposition
$$\Omega Y\simeq X^\prime \times R,$$
where $X^\prime$ is any irreducible factor such that $|X^\prime|=|\Omega Y|=|X|$. By Theorem \ref{Hilton}, we have a homotopy decomposition
$$\Omega \Sigma X \simeq \Omega (Y\vee Q)\simeq \Omega Y\times \Omega Q \times \prod_{\alpha\in \mathcal{I}-\{y_1, y_2\}}\Omega M((Y,\alpha_1), (Q,\alpha_2)).$$
Also, since $X$ is a $H$-space we have the usual decomposition
$$\Omega \Sigma X\simeq X\times \Omega \Sigma (X\wedge X).$$
Combining these decompositions, we see $X^\prime$ is a homotopy retract of $X\times \Omega \Sigma (X\wedge X)$. Since $|X^\prime|=|X|$, we have $X^\prime$ is a homotopy retract of $X$ by Lemma \ref{prime=irred} and \ref{prime}. Then $X\simeq X^\prime$ since $X$ is irreducible and we see that $X\simeq X^\prime$ is the only factor of $\Omega Y$ which has the minimal connectivity.

On the other hand, given any simply connected $p$-local irreducible co-$H$-space $Y$ of finite type, we have a decomposition
$$\Omega Y\simeq X \times R,$$
where $X$ is any irreducible factor such that $|X|=|\Omega Y|$, and then we have a decomposition of co-$H$-spaces
$$\Sigma X\simeq Y^\prime \vee Q$$
where $Y^\prime$ is any irreducible factor such that $|Y^\prime|=|\Sigma X|=|Y|$.
By Theorem \ref{suspension}, we have a homotopy decomposition
$$\Sigma \Omega Y \simeq Y\bigvee (\bigvee_{n=2}^{\infty}[\Sigma \Omega Y]_n)$$
with the property $\tilde{H}_\ast([\Sigma \Omega Y]_n) \cong \Sigma^{1-n}\tilde{H}_\ast(Y)^{\otimes n}$.
Also, we have the usual homotopy equivalence
$$\Sigma \Omega Y \simeq \Sigma (X \times R)\simeq \Sigma X\vee \Sigma R\vee \Sigma(X\wedge R).$$
Then combining the above decompositions, we see $Y^\prime$ is a homotopy retract of $Y\bigvee (\bigvee_{n=2}^{\infty}$ $[\Sigma \Omega Y]_n))$. Then again by Lemma \ref{prime=irred} and \ref{prime}, we have $Y^\prime$ is a homotopy retract of $Y$, and then $Y^\prime\simeq Y$ which is the only irreducible factor of $\Sigma X$ with the minimal connectivity.

The proof of theorem is completed by combining the above two parts.
\end{proof}

\begin{proposition}
Given a corresponding pair $(Y, X)$, we can choose structure maps $f$, $g$, $g^\prime$ and $h$ such that $g$ and $g^\prime$ are adjoint.
\end{proposition}
\begin{proof}
From the first part of the proof of Theorem \ref{correspondence}, we can choose $g$ as
$$X\stackrel{E}{\rightarrow}\Omega\Sigma X \stackrel{\Omega g^\prime}{\rightarrow} \Omega Y,$$
which immediately implies the claim.
\end{proof}

For the proof of our other main theorem (Theorem \ref{pGray} below) in this section, we recall the following suspension splitting of smash products of looped co-$H$-spaces.
\begin{theorem}[(\cite{Grbic}, Theorem $4.5$)]\label{ksuspension}
Let $X$ be a simply connected $p$-local co-$H$-space of finite type. Then there is natural homotopy decomposition
\begin{equation}\label{split}
\Sigma (\Omega X)^{\wedge k}\simeq \bigvee_{N=k}^{\infty} \bigvee^{\rho(N, k)} [\Sigma \Omega X]_N,
\end{equation}
where
\[
\rho(N,k)=\sharp \big\{ (n_1, n_2, \ldots, n_k)~|~n_i\geq 1, \forall~ i; ~n_1+n_2+\ldots+n_k=N \big\},
\]
and $[\Sigma \Omega X]_N$ is a space with the property that
\[
\tilde{H}_\ast([\Sigma \Omega X]_N) \cong \Sigma^{1-N}\tilde{H}_\ast(X)^{\otimes N}.
\]
\end{theorem}

\begin{theorem}\label{pGray}
Given a corresponding pair $(Y, X)$, there exists a homotopy equivalence
$$\Omega Y\simeq X\times \Omega W,$$
where $W$ is homotopy retract of $\Sigma(X\wedge X)$ and hence a co-$H$-space. Write $W\simeq \bigvee W_\alpha$ with each  $W_\alpha$ irreducible. Then $W_\alpha$ is homotopy retract of $[\Sigma\Omega Y]_n$ for some $n\geq 2$. In particular,
$\Sigma^{n-1}W_\alpha$ is a homotopy retract of $Y^{\wedge n}$. Further, if $Y=\Sigma Z$, $W_\alpha$ is homotopy retract of $\Sigma Z^{\wedge n}$.
\end{theorem}
\begin{proof}
The proof is almost the same as that of Theorem $3.2$ in \cite{Gray} which will be sketched here only for the purpose of convincing the readers that we do not need the $p$-complete setting and atomic condition. We start by constructing a diagram obtained  by taking the pullback of the Hopf fibration with the structure maps:
\[
  \begin{tikzcd}
  \Omega Y  \ar{r}{\Omega f}  \ar{d}{}      & \Omega\Sigma X  \ar{r}{\Omega g^\prime}  \ar{d}[swap]{\partial^\prime}
 &\Omega Y \ar{r}{\Omega f}  \ar{d}[swap]{h}   & \Omega\Sigma X  \ar{d}[swap]{\partial }\\
X  \ar[equal]{r}  \ar{d}{}                               & X  \ar[equal]{r}  \ar{d}{i^\prime}
 &X  \ar[equal]{r}  \ar{d}[swap]{\iota^\prime}                            & X    \ar{d}[swap]{i}                               \\
 W    \ar{r} \ar{d}[swap]{\pi^\prime} \ar[bend right, dashed]{rr}{id}                    & T\ar{r}\ar{d}
 & W   \ar{r}  \ar{d}[swap]{\pi}                               & \Sigma (X\wedge X)\ar{d}\\
 Y  \ar{r}{f}   \ar[bend right, dashed]{rr}{id}                          &\Sigma X \ar{r}{g^\prime}
 &Y \ar{r}{f}                                                    &\Sigma X.\\
  \end{tikzcd}
\]
By Proposition $A1$ in the Appendix of \cite{Gray}, $T$ is determined by the action map
$$X\times X\rightarrow \Omega \Sigma X\times X \rightarrow X.$$
Since $\partial^\prime= h\circ \Omega g^\prime$, the first component of the above map is $\partial^\prime\circ E=
h\circ \Omega g^\prime \circ E\simeq h\circ g\simeq id$. Then by Proposition $A2$ in the Appendix of \cite{Gray}, $T\simeq \Sigma(X\wedge X)$ and $i^\prime$ is null homotopic. Hence $W$ is a homotopy retract of $\Sigma(X\wedge X)$ and the $\iota^\prime$ is null homotopic which implies the required splitting $\Omega Y\simeq X\times \Omega W$.

For the remaining part, notice that $W_\alpha$ is a homotopy retract of $\Sigma(\Omega Y)^{\wedge 2}$. Also by Theorem \ref{ksuspension}, we have a decomposition
\[
\Sigma (\Omega Y)^{\wedge 2}\simeq \bigvee_{n=2}^{\infty} \bigvee^{\rho(n, 2)} [\Sigma \Omega Y]_n.
\]
Since $W_\alpha$ is prime, it is a homotopy retract of some  $[\Sigma \Omega Y]_n$ with $n\geq 2$ by finite type condition. Applying Theorem \ref{ksuspension} again, we have
\[
(\Sigma\Omega Y)^{\wedge n}\simeq \bigvee_{N=n}^{\infty} \bigvee^{\rho(N, n)} \Sigma^{n-1}[\Sigma \Omega Y]_N.
\]
Then the composition of maps
$$\Sigma^{n-1}[\Sigma \Omega Y]_n\rightarrow (\Sigma\Omega Y)^{\wedge n}\stackrel{{\rm ev}^{\wedge n}}\rightarrow Y^{\wedge n}$$
induces isomorphism of homology groups. Hence $\Sigma^{n-1}[\Sigma \Omega Y]_n\simeq Y^{\wedge n}$ which implies $\Sigma^{n-1}W_\alpha$ is a homotopy retract of $Y^{\wedge n}$. The last assertion follows from $[\Sigma\Omega \Sigma Z]_n\simeq \Sigma Z^{\wedge n}$ and we have completed the proof.
\end{proof}

\section{Homotopy rigidity of $\Sigma \Omega$ on co-$H$-spaces of finite type}
\noindent In \cite{Grbic}, Grbi{\'c} and the second author have proved the homotopy rigidity of $\Sigma\Omega$ for finite $p$-local co-$H$-spaces.
\begin{theorem}[(\cite{Grbic}, Theorem $4.7$)]
Let $X$ and $Y$ be simply connected $p$-local finite dimensional co-$H$-spaces, suppose that $\Sigma\Omega X\simeq \Sigma\Omega Y$, then $X\simeq Y$.
\end{theorem}

In this section, we want to generalize their result to finite type case:
\begin{theorem}\label{coHfinitetype}
Let $X$ and $Y$ be simply connected $p$-local co-$H$-spaces of finite type, suppose that $\Sigma\Omega X\simeq \Sigma\Omega Y$, then $X\simeq Y$.
\end{theorem}
\begin{proof} We may prove the theorem by inductions on both the connectivity and numbers of the irreducible components with the same connectivity. First, by Theorem \ref{suspension} or \ref{ksuspension}, we have
$$\Sigma\Omega X \simeq X\bigvee \bigvee_{i=2}^{\infty} [ \Sigma\Omega X ]_n.$$
Since $\Sigma\Omega X\simeq \Sigma\Omega Y$, then we have
$$X\bigvee \bigvee_{i=2}^{\infty}[ \Sigma\Omega X ]_n\simeq Y\bigvee \bigvee_{i=2}^{\infty}[ \Sigma\Omega Y ]_n.$$
We define a homotopy functor $\mathcal{M}$ by $\mathcal{M}(X)\simeq \bigvee_{i=2}^{\infty} [ \Sigma\Omega X ]_n$. Then the above homotopy equivalence can be written as
\begin{equation}\label{coHfinitetype1}
X\vee \mathcal{M}(X)\simeq Y\vee\mathcal{M}(Y).
\end{equation}
Now we want to study the property of $\mathcal{M}$. By Theorem \ref{Hilton}, we have for any simply connected $p$-local co-$H$-spaces $W_1$ and $W_2$ of finite type
\begin{eqnarray*}
\Sigma \Omega (W_1\vee W_2)
&\simeq& \Sigma \big(  \Omega W_1\times \Omega W_2 \times \prod_{\alpha\in \mathcal{I}-\{y_1, y_2\}}\Omega M((W_1,\alpha_1), (W_2,\alpha_2))\big)\\
&\simeq& \Sigma \Omega W_1 \vee \Sigma \Omega W_2 \vee \mathcal{M}(W_1, W_2)\\
&\simeq& W_1\vee \mathcal{M}(W_1)\vee W_2\vee \mathcal{M}(W_2)\vee \mathcal{M}(W_1, W_2),\\
\end{eqnarray*}
where $\mathcal{M}(W_1, W_2)$ is a suitable homotopy bi-functor defined by the above deduction, and is said to be \textit{reduced with respect to $W_1$ and $W_2$} in the sense that

\textit{$\mathcal{M}(W_1, W_2)\simeq \ast$ if either $W_1$ or $W_2$ is homotopy contractible.}

On the other hand, we have
$$\Sigma \Omega (W_1\vee W_2)\simeq  W_1\vee W_2\vee \mathcal{M}(W_1\vee W_2).$$
Then by Theorem \ref{finitetype}, we have
\begin{equation}\label{coHfinitetype2}
\mathcal{M}(W_1\vee W_2)\simeq  \mathcal{M}(W_1)\vee \mathcal{M}(W_2)\vee \mathcal{M}(W_1, W_2).
\end{equation}
Another important property of the bi-functor $\mathcal{M}(-,-)$ is that it is \textit{splittable in both entries} in the sense of the following:

\textit{Given any three co-$H$-space $W_i$ for $i=1$, $2$, $3$ of mentioned type, we have
\begin{equation}\label{coHfinitetype3}
\mathcal{M}(W_1, W_2\vee W_3)\simeq \mathcal{M}(W_1, W_3)\vee \mathcal{M}(W_1,W_2; W_3),
\end{equation}
for some tri-functor $\mathcal{M}(-,-; -)$, and similar decomposition holds for $\mathcal{M}(W_1\vee W_2, W_3)$.}

Furthermore, $\mathcal{M}(W_1,W_2; W_3)$ is reduced and splittable with respect to $W_1$ and $W_2$.

Now return to our proof of the theorem. We may decompose $X$ as
$$X\simeq Z_1\vee X_1,$$
such that $Z_1$ is irreducible and $|Z_1|=|X|=|Y|<|\mathcal{M}(Y)|$. Then by Lemma \ref{prime=irred} and \ref{prime}, we see from (\ref{coHfinitetype1}) that $Z_1$ is a homotopy retract of $Y$ which implies
$$Y\simeq Z_1\vee Y_1,$$
Then
$$\Sigma \Omega X\simeq \Sigma \Omega (Z_1\vee X_1)\simeq Z_1\vee X_1\vee \mathcal{M}(Z_1)\vee \mathcal{M}(X_1)\vee \mathcal{M}(Z_1, X_1),$$
$$\Sigma \Omega Y\simeq \Sigma \Omega (Z_1\vee Y_1)\simeq Z_1\vee Y_1\vee \mathcal{M}(Z_1)\vee \mathcal{M}(Y_1)\vee \mathcal{M}(Z_1, Y_1).$$
By Theorem \ref{finitetype}, we see
\begin{equation}\label{coHfinitetype4}
X_1\vee \mathcal{M}(X_1)\vee \mathcal{M}(Z_1, X_1)\simeq Y_1\vee \mathcal{M}(Y_1)\vee \mathcal{M}(Z_1, Y_1),
\end{equation}
which implies $|X_1|=|Y_1|$. Suppose we have further decomposition
$$X_1\simeq Z_2\vee X_2,$$
such that $Z_2$ is irreducible and $|Z_2|=|X_1|=|Y_1|<|\mathcal{M}(Y_1)|<|\mathcal{M}(Z_1, Y_1)|$. Then we see $Z_2$ is a homotopy retract of $Y_1$ which implies
$$Y_1\simeq Z_2\vee Y_2.$$
Now (\ref{coHfinitetype4}) becomes
$$Z_2\vee X_2\vee \mathcal{M}(Z_2\vee X_2)\vee \mathcal{M}(Z_1, Z_2\vee X_2)\simeq Z_2\vee Y_2\vee \mathcal{M}(Z_2\vee Y_2)\vee \mathcal{M}(Z_1, Z_2\vee Y_2),$$
which with the help of (\ref{coHfinitetype2}) and (\ref{coHfinitetype3}) reduces to
\begin{equation}\label{coHfinitetype5}
 X_2\vee  \mathcal{M}(X_2)\vee \mathcal{M}(Z_2, X_2)
\vee \mathcal{M}(Z_1,X_2; Z_2)\simeq  Y_2\vee  \mathcal{M}(Y_2)\vee \mathcal{M}(Z_2, Y_2)
\vee  \mathcal{M}(Z_1,Y_2; Z_2).
\end{equation}
We may define a homotopy functor $\mathcal{M}_2$ by
$$\mathcal{M}_2((Z_1,Z_2), X_2)\simeq \mathcal{M}(Z_2, X_2) \vee \mathcal{M}(Z_1,X_2; Z_2).$$
On the one hand, since $\mathcal{M}(Z_2, X_2)$ is reduced and $\mathcal{M}(Z_1,X_2; Z_2)$ is reduced with respect to $Z_1$ and $X_2$, we see $\mathcal{M}_2((Z_1,Z_2), X_2)$ is reduced with respect to $(Z_1, Z_2)$ and $X_2$ (i.e., $\mathcal{M}_2((Z_1,Z_2), X_2)\simeq \ast$ if either $(Z_1, Z_2)\simeq (\ast,\ast )$ or $X_2\simeq \ast$). On the other hand, since both $\mathcal{M}(Z_2, X_2)$ and $\mathcal{M}(Z_1,X_2; Z_2)$ are splittable with respect to $X_2$, $\mathcal{M}_2((Z_1,Z_2), X_2)$ is splittable respect to $X_2$. To summarize, we have decompositions
$$X\simeq Z_1\vee Z_2\vee X_2,$$
\begin{equation}\label{coHfinitetype6}
Y\simeq Z_1\vee Z_2\vee Y_2,
\end{equation}
\[
X_2\vee  \mathcal{M}(X_2)\vee \mathcal{M}_2((Z_1,Z_2), X_2)\simeq  Y_2\vee  \mathcal{M}(Y_2)\vee\mathcal{M}_2((Z_1,Z_2), Y_2)
\]
such that $\mathcal{M}_2((W_1,W_2), W_3)$ is reduced w.r.t. $(W_1,W_2)$ and $W_3$, and is splittable w.r.t. $W_3$.

Now by induction, suppose we have decompositions
$$X\simeq Z_1\vee Z_2\vee \cdots\vee Z_{n-1}\vee X_{n-1},$$
\begin{equation}\label{coHfinitetype7}
Y\simeq Z_1\vee Z_2\vee \cdots\vee Z_{n-1}\vee Y_{n-1},
\end{equation}
\[
X_{n-1}\vee  \mathcal{M}(X_{n-1})\vee \mathcal{M}_{n-1}((Z_i)_{i=1}^{n-1}, X_{n-1})\simeq  Y_{n-1}\vee  \mathcal{M}(Y_{n-1})\vee \mathcal{M}_{n-1}((Z_i)_{i=1}^{n-1}, Y_{n-1})
\]
such that $|X|=|Z_1|\leq |Z_2|\leq \ldots \leq |Z_{n-1}|\leq |X_{n-1}|$, and also $\mathcal{M}_{n-1}((W_i)_{i=1}^{n-1}, W_n)$ is reduced w.r.t. $(W_i)_{i=1}^{n-1}$ and $W_n$, and is splittable w.r.t. $W_n$.
Now $X_{n-1}$ can be further decomposed as
$$X_{n-1}\simeq Z_n\vee X_n,$$
such that $X_n$ is irreducible and
$$|Z_n|=|X_{n-1}|=|Y_{n-1}|<|\mathcal{M}(Y_{n-1})|<|\mathcal{M}_{n-1}(((Z_i)_{i=1}^{n-1}, Y_{n-1})|.$$
Hence again we have
$$Y_{n-1}\simeq Z_n\vee Y_n,$$
and (\ref{coHfinitetype7}) reduces to
\begin{eqnarray*}
X_{n}\vee  \mathcal{M}(X_n)\vee \mathcal{M}(Z_n, X_n)\vee \mathcal{M}_{n-1}((Z_i)_{i=1}^{n-1}, Z_n\vee X_n)\\
\simeq Y_{n}\vee  \mathcal{M}(Y_n)\vee \mathcal{M}(Z_n, Y_n)\vee \mathcal{M}_{n-1}((Z_i)_{i=1}^{n-1}, Z_n\vee Y_n).
\end{eqnarray*}
Since $\mathcal{M}_{n-1}$ is splittable with respect to the last entry, we can write
\[
\mathcal{M}_{n-1}((Z_i)_{i=1}^{n-1}, Z_n\vee X_n)\simeq \mathcal{M}_{n-1}((Z_i)_{i=1}^{n-1}, Z_n)\vee \mathcal{M}_{n-1}((Z_i)_{i=1}^{n-1}, X_n; Z_n),
\]
define $\mathcal{M}_{n}$ by
\[
\mathcal{M}_{n}((Z_i)_{i=1}^{n}, X_n) \simeq \mathcal{M}(Z_n, X_n) \vee \mathcal{M}_{n-1}((Z_i)_{i=1}^{n-1}, X_n; Z_n),
\]
which is clearly reduced with respect to $(Z_i)_{i=1}^{n}$ and $X_n$.
Since all the $\mathcal{M}_{i}$'s and other involved functors are derived from the generalized Hilton-Milnor Theorem (Theorem \ref{Hilton}), we see that $\mathcal{M}_{n}$ is splittable with respect to the last entry (which can be deduced from the decomposition of a looped wedge of $n+2$ co-$H$-spaces. Roughly speaking, $\mathcal{M}_{n}((Z_i)_{i=1}^{n}, X_n)$ consists of the part of $\Sigma\Omega (\vee_{i=1}^{n}Z_i\vee X_n)$ ``containing'' $X_n$ and at least one of $Z_i$'s, which also justifies the chosen of the name ``reduced''). Then (\ref{coHfinitetype7}) can be further reduced to
\[
X_{n}\vee  \mathcal{M}(X_{n})\vee \mathcal{M}_{n}((Z_i)_{i=1}^{n}, X_{n})\simeq  Y_{n}\vee  \mathcal{M}(Y_{n})\vee \mathcal{M}_{n}((Z_i)_{i=1}^{n}, Y_{n}),
\]
which completes the induction step. Finally, we notice that the given spaces are of finite type, and then an induction argument on the connectivity will show that $X$ and $Y$ have the same irreducible components of any connectivity. By the unique decomposition theorem (Theorem \ref{finitetype}), $X\simeq Y$.
\end{proof}

The following corollary follows immediately.
\begin{corollary}
Let $X$ and $Y$ be simply connected $p$-local co-$H$-spaces of finite type, suppose that $\Omega X\simeq \Omega Y$, then $X\simeq Y$.
\end{corollary}

\section{Homotopy rigidity of $\Omega\Sigma$ and its functorial retract on $H$-spaces of finite type}
\noindent In this section we will study the homotopy rigidity problem for some good functorial retracts of $\Omega\Sigma$. First, we should clarify our meaning of good.
\begin{definition}\label{goodretract}
Given any functorial coalgebra decomposition
\[
T(V)\cong A(V)\otimes B(V),
\]
it is said to be \textit{good} if $A_1(V)\cong V$ and the natural morphism $B(V)\hookrightarrow T(V)$ is a functorial injection of Hopf algebras. Then the geometric realization of the above decomposition (the existence of which is ensured by \cite{Selick} and \cite{Selick2})
\[
\Omega\Sigma X\simeq A(X)\times B(X)
\]
 for any connected $p$-local space $X$ is called \textit{a good natural (or functorial) decomposition of} $\Omega \Sigma$, and $A$ is called  \textit{a functorial retract of} $\Omega\Sigma$ \textit{over identity}.
\end{definition}

\begin{remark}
There are examples of good functorial decompositions. Indeed, the functorial coalgebra decomposition 
\[T(V)\cong  A^{{\rm min}}(V) \otimes  B^{{\rm max}} (V)\] 
of Selick and the second author \cite{Selick,Selick1} is the motivation of the above definition. On the contrary, there are decompositions which are not good. In \cite{Selick}, Selick and the second author proved a \textit{functorial Poincar\'{e}-Birkhoff-Witt theorem} which claims that there is a functorial coalgebra isomorphism
\[T(V)\cong \bigotimes_{n=1}^{\infty}A^{{\rm min}}(V; L_n^{{\rm max}}),\]
where $A^{{\rm min}}(V; L_n^{{\rm max}})$ is the minimal functorial coalgebra retract of $T(V)$ over $L_n^{{\rm max}}(V)$, and $L_n^{{\rm max}}(V)$ is a certain natural submodule of $V^{\otimes n}$ and $L_1^{{\rm max}}(V)\cong V$. Hence in order to get a decomposition which is not good, we only need to choose the tensor of suitable terms $A^{{\rm min}}(V; L_n^{{\rm max}})$ $(n\geq 2)$ as $B(V)$ such that it is not a sub-Hopf algebra of $T(V)$, which should be true in the most cases. For instance, let $B(V)=A^{{\rm min}}(V; L_n^{{\rm max}})$ for some $n\geq 2$.
\end{remark}

Notice that here we only consider a special case of the functorial decomposition of loops on co-$H$-spaces in \cite{Selick2}. Under the condition of the above definition, we have
\[
B(V)\cong T(Q(V)),
\]
by Theorem $8.8$ of \cite{Selick}, where
$$Q(V)\cong \oplus_{n=2}^{\infty} Q_n(V)$$
is the set of decomposable elements of $B(V)$ and decomposed with respect to tensor length. Further, $Q_n(V)$ is a functorial retract of $L_n(V)$ which is the $n$-th homogeneous component of the free Lie algebra $L(V)$ generated by $V$. Then $Q_n(V)$ is $T_n$-projective and corresponds to a $\mathbb{Z}/p[\Sigma_n]$-projective submodule $Q_n$ of ${\rm Lie}(n)$ by Proposition $2.5$ of \cite{Wu} (for details, one can check subsection $2.2$ of \cite{Wu}), and also
\begin{equation}\label{ungradedtograd}
Q_n(V)\cong Q_n\otimes_{\mathbb{Z}/p[\Sigma_n]} V^{\otimes n}.
\end{equation}
Since $Q_n$ is a functorial retract of $T_n$, then it is the image of a functorial morphism
\[
f_V: V^{\otimes n}=T_n(V)\rightarrow T_n(V)=V^{\otimes n},
\]
and by Lemma $2.1$ of \cite{Selick}, $f_V\in \mathbb{Z}/p[\Sigma_n]$. Then we can define
\[
\tilde{f}_X: \Sigma X^{\wedge n}\rightarrow \Sigma X^{\wedge n},
\]
as the realization of $f_V$ such that $\mathbb{Z}_{(p)}[\Sigma_n]$ acts on $\Sigma X^{\wedge n}$ by permuting factors and using the comultiplication on $\Sigma X^{\wedge n}$. We then define
\[
Q_n(X)={\rm hocolim}_{\tilde{f}_X} \Sigma X^{\wedge n}
 \]
 for any $n\geq 2$. It turns out that $Q_n(X)$ is the functorial geometric realization of $Q_n(V)$ with $V\cong \tilde{H}_\ast(X)$, and is also a functorial homotopy retract of $\Sigma X^{\wedge n}$ (Lemma $2.2$ of \cite{Selick1}). Hence as in \cite{Selick1}, we choose a functorial cross-section $\theta _n: Q_n(X)\rightarrow \Sigma X^{\wedge n}$ for each $n$, and define the following composition of maps
 \[
 \phi: Q(X):=\bigvee_{n=2}^{\infty} Q_n(X)\stackrel{\vee \theta_n}{\longrightarrow}  \bigvee_{n=2}^{\infty} \Sigma X^{\wedge n} \stackrel{\vee \omega_n}{\longrightarrow} \bigvee\Sigma X\stackrel{\nabla}{\longrightarrow}\Sigma X,
 \]
 where $\nabla$ is the folding map and $\omega_n$ is the $n$-fold Whitehead product of identity map on $\Sigma X$ with itself. Then we get a functorial fibre sequence
\begin{equation}\label{fibreprime}
\Omega(Q(X))\rightarrow \Omega \Sigma X\rightarrow A^\prime (X)\rightarrow Q(X)\stackrel{\phi}{\longrightarrow}\Sigma X.
\end{equation}
Notice that the composition of natural maps
\[
\Omega(Q(X))\rightarrow \Omega \Sigma X\rightarrow B(X),
\]
induces an isomorphism on the submodule $ Q(V)\cong \Sigma^{-1}\tilde{H}_\ast(Q(X))$ of $\tilde{H}_\ast(B(X))\cong B(V)\cong T(Q(V))$, it is then a homotopy equivalence, i.e.,
\[
\Omega(Q(X))\simeq B(X).
\]
Then the first part of (\ref{fibreprime}) splits as
\[
\Omega \Sigma X\simeq \Omega(Q(X))\times A^\prime (X).
\]
Then by unique decomposition theorem (Theorem \ref{finitetype}), we see that $A^\prime(X)\simeq A(X)$ which is indeed a functorial homotopy equivalence. Hence (\ref{fibreprime}) can be chosen to be
\begin{equation}\label{fibre}
\Omega(Q(X))\rightarrow \Omega \Sigma X\rightarrow A (X)\rightarrow Q(X)\stackrel{\phi}{\longrightarrow}\Sigma X.
 \end{equation}

 We now want to study a special splittable property of $Q_n$. Suppose $X\simeq X_1\times X_2$, then
 \[
 \Sigma X\simeq \Sigma X_1\vee \Sigma X_2\vee \Sigma (X_1\wedge X_2),
 \]
 which implies
 \begin{eqnarray*}
 \Sigma X^{\wedge n}
 &\simeq&  \Sigma \big( X_1\vee X_2\vee  (X_1\wedge X_2)\big) \wedge X^{\wedge (n-1)}\\
 &\simeq&  \Sigma \big( X_1\vee X_2\vee  (X_1\wedge X_2)\big) ^{\wedge n}\\
 &\simeq&  \Sigma X_1^{\wedge n}\vee \Sigma X_2^{\wedge n}\vee \Sigma \mathcal{S}_n(X_1, X_2),
 \end{eqnarray*}
 where $\mathcal{S}_n(X_1, X_2)$ is a homotopy bi-functor defined by the above calculation. It is clear that $\Sigma\mathcal{S}_n$ is reduced and splittable  in both entries, i.e., $\Sigma\mathcal{S}_n(X_1, X_2)\simeq \ast$ if $X_1\simeq \ast$ or $X_2\simeq \ast$,
 and $\Sigma\mathcal{S}_n(X_1, X_2\times X_3)\simeq \Sigma\mathcal{S}_n(X_1, X_2)\vee \Sigma\mathcal{S}_n(X_1, X_3; X_2)$ for some tri-functor $\mathcal{S}_n(X_1, X_3; X_2)$ (it should be noticed that the operation involved in the definition of splittable may vary according to the context). Then the corresponding decomposition of the above one on the algebraic level should be
 \begin{equation}\label{tensordecom}
 \big(W_1\oplus W_2 \oplus (W_1\otimes W_2)\big)^{\otimes n} \cong  W_1^{\otimes n}\oplus W_1^{\otimes n} \oplus\mathcal{T}_n (W_1, W_2),
 \end{equation}
where $\mathcal{S}_n$ is the geometric realization of $\mathcal{T}_n$. Denote $V^{+}\cong \mathbb{Z}/p\oplus V$, $\widetilde{V}^{+}=V$ and $V^{+}\cong W_1^{+}\otimes W_2^{+}$, we see $V\cong W_1\oplus W_2 \oplus (W_1\otimes W_2)$. Now since $f_V\in \mathbb{Z}/p[\Sigma_n]$ and (\ref{tensordecom}) is a $\Sigma_n$-invariant decomposition,
we  see
\[
f_V\cong f_{W_1}\oplus f_{W_2}\oplus f_{W_1, W_2},
\]
where $f_{W_1, W_2}$ is some suitable functorial retraction of $\mathcal{T}_n (W_1, W_2)$, and then determine a functorial retract $\mathcal{Q}_n(W_1, W_2)$ of $\mathcal{T}_n (W_1, W_2)$. Hence there exists a functorial decomposition
\begin{equation}\label{decomQnV}
Q_n\big(V\cong (W_1^{+}\otimes W_2^{+})^{\sim}\big) \cong Q_n(W_1)\oplus Q_n(W_2)\oplus \mathcal{Q}_n(W_1, W_2),
\end{equation}
which can be explicitly described by (\ref{ungradedtograd}) (which is only stated for ungraded modules in \cite{Wu}, but it can be generalized to the graded case by using Lemma $3.2$ of \cite{Selick} in our situation). Then $\mathcal{Q}_n(W_1, W_2)$ is also reduced and splittable.

Returning to the geometric level, we see that there exists a functorial retraction
\[
Q_n(X_1)\vee Q_n(X_2)\rightarrow Q_n(X\simeq X_1\times X_2),
\]
the cofibre of which may be denoted by $\mathcal{Q}_n(X_1, X_2)$. Hence $\mathcal{Q}_n(X_1, X_2)$ is a geometric realization of $\mathcal{Q}_n(W_1, W_2)$ with $W_i\cong \tilde{H}(X_i)$ for $i=1$, $2$ and is reduced and splittable as a functorial retract of $\Sigma \mathcal{S}_n(X_1, X_2)$. Thus we have
\begin{eqnarray*}
B(X\simeq X_1\times X_2)
&\simeq& \Omega \big(\bigvee_{n=2}^{\infty}Q_n(X_1\times X_2)\big) \\
&\simeq& \Omega \big(\bigvee_{n=2}^{\infty}Q_n(X_1)\vee\bigvee_{n=2}^{\infty}Q_n(X_2) \vee \bigvee_{n=2}^{\infty}\mathcal{Q}_n(X_1, X_2)\big)\\
&\simeq& B(X_1)\times B(X_2)\times \mathcal{B}(X_1, X_2),
\end{eqnarray*}
where the last step is obtained by Theorem \ref{Hilton}, and $\mathcal{B}(X_1, X_2)$ is some suitable bi-functor which is also reduced and splittable. We also have
\begin{eqnarray*}
\Omega \Sigma (X\simeq X_1\times X_2)
&\simeq& \Omega (\Sigma X_1\vee \Sigma X_2\vee \Sigma (X_1\wedge X_2))\\
&\simeq& \Omega \Sigma X_1\times \Omega \Sigma X_2 \times \mathcal{J}(X_1, X_2)
\end{eqnarray*}
for some bi-functor $\mathcal{J}(X_1, X_2)$. It is noticed that $\mathcal{B}(X_1, X_2)$ is a functorial homotopy retract of $\mathcal{J}(X_1, X_2)$, and we have a natural decomposition
\begin{equation}\label{decompA}
A(X\simeq X_1\times X_2)\simeq A(X_1)\times A(X_2)\times \mathcal{A}(X_1, X_2),
\end{equation}
for some bi-functor $\mathcal{A}(X_1, X_2)$ which is a functorial homotopy retract of $\mathcal{J}(X_1, X_2)$ and is also reduced and splittable. Now we can prove the main theorem in this section.
\begin{theorem}\label{A}
Let $X$ and $Y$ be connected $p$-local $H$-spaces of finite type, and $A(X)$ be any good functorial homotopy retract of $\Omega\Sigma X$ over $X$. Then if $A(X)\simeq A(Y)$, we have $X\simeq Y$.
\end{theorem}
\begin{proof} First by \ref{decompA}, we have a functorial decomposition
\[
\Sigma A(X_1\times X_2)\simeq \Sigma A(X_1)\vee \Sigma A(X_2)\vee \mathcal{E}_A(X_1, X_2),
\]
for some bi-functor $\mathcal{E}_A(X_1, X_2)$ which is reduced and splittable. By Theorem \ref{suspension}, we have
$$\Omega \Sigma A(X)\simeq \Omega\big(\Sigma X\vee \bigvee_{n=2}^{\infty} A_n(X)\big)\simeq X\times \mathcal{J}_A(X),$$
for some suitable functor $\mathcal{J}_A(X)$ such that $|\mathcal{J}_A(X)|>|X|$ since $X$ is an $H$-space. Then
\begin{eqnarray*}
\Omega\Sigma A(X_1\times X_2)
&\simeq& \Omega \big(\Sigma A(X_1)\vee \Sigma A(X_2)\vee \mathcal{E}_A(X_1, X_2)\big)\\
&\simeq& \Omega \Sigma A(X_1)\times\Omega \Sigma A(X_2)\times \mathcal{J}_A(X_1, X_2)\\
&\simeq& X_1\times \mathcal{J}_A(X_1)\times X_2\times \mathcal{J}_A(X_2)\times \mathcal{J}_A(X_1, X_2),
\end{eqnarray*}
for some reduced and splittable bi-functor $\mathcal{J}_A(X_1, X_2)$, which implies
\[
\mathcal{J}_A(X_1\times X_2)\simeq \mathcal{J}_A(X_1)\times \mathcal{J}_A(X_2)\times \mathcal{J}_A(X_1, X_2).
\]
Now suppose there exists a decomposition
\[
X\simeq Z_1\times X_1,
\]
such that $Z_1$ is irreducible and $|Z_1|=|X|=|Y|$. Since $\Omega\Sigma A(X)\simeq \Omega\Sigma A(Y)$ by assumption, we have
\[
Z_1\times X_1 \times\mathcal{J}_A(Z_1)\times \mathcal{J}_A(X_1)\times \mathcal{J}_A(Z_1, X_1)\simeq Y\times \mathcal{J}_\mathcal{A}(Y).
\]
Then $Z_1$ is a homotopy retract of $Y$ by connectivity and Lemma \ref{prime}. Hence there exists a decomposition
\[
Y\simeq Z_1\times Y_1,
\]
which implies
\[
X_1 \times \mathcal{J}_A(X_1)\times \mathcal{J}_A(Z_1, X_1)\simeq Y_1 \times \mathcal{J}_A(Y_1)\times \mathcal{J}_A(Z_1, Y_1).
\]
The theorem is then can be proved by a similar but dual argument to that of Theorem \ref{coHfinitetype}, by the observation that any multi-functor involved in the induction process will be splittable and reduced in a similar fashion to that in the proof of Theorem \ref{coHfinitetype} (for the decompositions are all $\Sigma_n$-invariant, and then similar types of decompositions hold for functorial retracts as we discussed before the proof of the theorem).
\end{proof}

For the special case when $A=\Omega \Sigma$, Theorem \ref{A} serves as the dual version of Theorem \ref{coHfinitetype} which also confirms the conjecture raised in \cite{Grbic}.

\begin{theorem}\label{Hfinitetype}
Let $X$ and $Y$ be connected $p$-local $H$-spaces of finite type, suppose that $\Omega \Sigma X\simeq \Omega \Sigma  Y$, then $X\simeq Y$.
\end{theorem}
The following corollary then follows immediately.
\begin{corollary}
Let $X$ and $Y$ be connected $p$-local $H$-spaces of finite type, suppose that $\Sigma X\simeq \Sigma  Y$, then $X\simeq Y$.
\end{corollary}

\section{Appendix}
\noindent In this appendix we discuss the rigidity problem of some other canonical homotopy functors. First, for any simply connected $p$-local co-$H$-space $Y$ of finite type, we have self-wedge functor $\vee_n$ for each positive integer $n$ such that
$$\vee_n(Y)= \underbrace{Y\vee  \ldots \vee Y}_{n}.$$
And similarly for any connected $p$-local $H$-space $X$ of finite type, we have self-product functor $\times_m$ for each $m$ such that
$$\times_m(X)= \underbrace{X\times  \ldots \times X}_{m}.$$
The following proposition follows immediately from Theorem \ref{finitetype}.
\begin{proposition}
$\vee_n$ and $\times_m$ are homotopy rigid.
\end{proposition}
We may also define self-smash functor $\wedge_n$ for any simply connected $p$-local co-$H$-space $Y$ of finite type by
$$\wedge_n(Y)= \underbrace{Y\wedge  \ldots \wedge Y}_{n}.$$
However, $\wedge_n$ is not homotopy rigid.
\begin{proposition}
There exist some finite simply connected $p$-local co-$H$-space $X$ and $Y$ such that $\wedge_n(X)\simeq \wedge_n(Y)$ but $X\not\simeq Y$.
\end{proposition}
\begin{proof}
Choose $X\simeq \Sigma (S^n\cup_\alpha S^m)$ and $Y\simeq S^{n+1}\vee S^{m+1}$ such that $\Sigma \alpha\not\simeq \ast$ but $\Sigma^{2}\alpha\simeq \ast$. Then we have
\begin{eqnarray*}
\wedge_2(X)\simeq X\wedge X
&\simeq& (\Sigma S^n\cup_\alpha S^m)\wedge (\Sigma S^n\cup_\alpha S^m)\\
&\simeq& (\Sigma^{2} S^n\cup_\alpha S^m)\wedge (S^n\cup_\alpha S^m)\\
&\simeq& \Sigma^{2}( S^{n}\vee S^{m})\wedge (S^n\cup_\alpha S^m)\\
&\simeq& ( S^{n}\vee S^{m})\wedge \Sigma^{2}( S^{n}\vee S^{m})\\
&\simeq& Y\wedge Y\simeq\wedge_2(Y).
\end{eqnarray*}
Of course, there exists such $\alpha$. For instance, for odd prime $p$, there exists
$$\alpha_1(n): S^{2n+2p-3}\rightarrow S^n$$
for $n\geq 3$ such that $\alpha_1(3)\circ \alpha_1(2p)$ and $\alpha_1(4)\circ \alpha_1(2p+1)$ are essential, but $\alpha_1(n)\circ \alpha_1(n+2p-3)$ is not for $n\geq 5$ \cite{Harper}. For $p=2$, there exist Hopf elements
$$\eta_2: S^3\rightarrow S^2, \ \ \ \quad \nu_4: S^7\rightarrow S^4,$$
and $\eta_n=\Sigma^{n-2} \eta_2$, $\nu_n=\Sigma^{n-4} \nu_4$. Then the compositions $\eta_3\circ \nu_4$ and $\eta_4\circ \nu_5$ are essential, but $\eta_n\circ \nu_{n+1}$ is not for $n\geq 5$ \cite{Toda}.
\end{proof}

There is also a semi-product operation $\ltimes$ defined by $X\ltimes Y = X\times Y /X\times \{\ast\}$. Therefore we may define a semi-product functor $X\ltimes$ for any simply connected $p$-local co-$H$-space $X$ of finite type such that
$$X\ltimes(Y)\simeq X\ltimes Y.$$
Then we have the following proposition.
\begin{proposition}
$X\ltimes$ is homotopy rigid as a functor from the category of simply connected $p$-local co-$H$-spaces of finite type to the category of spaces.
\end{proposition}
\begin{proof}
We prove the proposition by inductions on both the connectivity and the number of irreducible components of the same connectivity. Given any $Y$ and $W$ of required type such that
$$X\ltimes Y\simeq X\ltimes W,$$
we want to prove $Y\simeq W$.
Since $X$ is a co-$H$-space, then we see $X\ltimes Y\simeq Y\vee (X\wedge Y)$ which implies
$$Y\vee (X\wedge Y)\simeq W\vee (X\wedge W).$$
If $Y$ is irreducible, then $Y$ is a homotopy retract of $W$ by Lemma \ref{prime=irred} and \ref{prime} which implies $Y\simeq W$. Now suppose we have
$$Y\simeq Z_{(1)}\vee Z_{(2)}\vee \ldots\vee Z_{(n-1)}\vee \widetilde{Y}_{(n)},$$
$$W\simeq Z_{(1)}\vee Z_{(2)}\vee \ldots\vee Z_{(n-1)}\vee \widetilde{W}_{(n)},$$
$$\widetilde{Y}_{(n)} \vee (X\wedge \widetilde{Y}_{(n)}) \simeq \widetilde{W}_{(n)}\vee (X\wedge \widetilde{W}_{(n)}),$$
where each $ Z_{(i)}$ is irreducible and $|Y|=|Z_{(1)}|\leq |Z_{(2)}|\leq \ldots\leq |Z_{(n-1)}|\leq |\widetilde{Y}_{(n)}|$. Make further decomposition as
$$\widetilde{Y}_{(n)}\simeq Z_{(n)}\vee \widetilde{Y}_{(n+1)},$$
such that $Z_{(n)}$ is irreducible and $|Z_{(n)}|=|\widetilde{Y}_{(n)}|=|\widetilde{W}_{(n)}|$. Then $Z_{(n)}$ is a homotopy retract of $\widetilde{W}_{(n)}\vee (X\wedge \widetilde{W}_{(n)})$, and then of $\widetilde{W}_{(n)}$ by Lemma \ref{prime}. Hence $\widetilde{W}_{(n)}\simeq Z_{(n)}\vee \widetilde{W}_{(n+1)}$, and we have
$$Y\simeq Z_{(1)}\vee Z_{(2)}\vee \ldots\vee Z_{(n-1)}\vee Z_{(n)}\vee \widetilde{Y}_{(n+1)},$$
$$W\simeq Z_{(1)}\vee Z_{(2)}\vee \ldots\vee Z_{(n-1)}\vee Z_{(n)}\vee \widetilde{W}_{(n+1)},$$
$$Z_{(n)}\vee \widetilde{Y}_{(n+1)} \vee (X\wedge Z_{(n)})\vee (X\wedge \widetilde{Y}_{(n+1)}) \simeq Z_{(n)}\vee \widetilde{W}_{(n+1)} \vee (X\wedge Z_{(n)})\vee (X\wedge \widetilde{W}_{(n+1)}),$$
where the last equivalence can be simplified to be
$$\widetilde{Y}_{(n+1)} \vee (X\wedge \widetilde{Y}_{(n+1)}) \simeq \widetilde{W}_{(n+1)} \vee (W\wedge \widetilde{W}_{(n+1)}).$$
Hence the induction step is completed, and then $Y$ and $W$ have the same irreducible components of any connectivity. By the unique decomposition theorem (Theorem \ref{finitetype}), $Y\simeq W$, and we have proved the proposition.
\end{proof}

To conclude the paper, let us consider the free loop functor $L$ which is the basic object in string topology. By definition, $L(X)={\rm Map}(S^1, X)$, the un-based mapping space. Also, we define a space $X$ of finite type to be \textit{homotopy finite} if $\pi_n(X)=0$ for all but finitely many $n$. Now for the rigidity problem of $L$, we have the following proposition.
\begin{proposition}
$L$ is homotopy rigid as a functor from the category of simply connected homotopy finite $p$-local $H$-spaces of finite type to the category of spaces.
\end{proposition}
\begin{proof}
Given any simply connected homotopy finite $p$-local $H$-spaces $X$ and $Y$ of finite type such that $L(X)\simeq L(Y)$, we want to prove $X\simeq Y$.
Recall that there is a canonical fibration
$$\Omega X \rightarrow L(X) \rightarrow X,$$
with a canonical cross-section $X\rightarrow L(X)$. Since $X$ is an $H$-space, we see $L(X)\simeq X\times \Omega X$. Hence,
\begin{equation}\label{App1}
X\times \Omega X\simeq  Y\times \Omega Y.
\end{equation}
Now we want to prove the proposition by induction on the number of irreducible components of $X$. First, suppose $X$ is irreducible as an $H$-space. Then $X$ is a homotopy retract of $Y$ or $\Omega Y$ by Lemma \ref{prime}. If $X$ is a retract of $Y$, then it is easy to see $X\simeq Y$. Otherwise, $X$ is a retract of $\Omega Y$ which implies $\Omega Y\simeq X\times W$ for some $H$-space $W$. Then (\ref{App1}) becomes
$$X\times \Omega X\simeq  Y\times X\times W,$$
which implies
$$\Omega X\simeq Y\times W$$
by Theorem \ref{finitetype}. Hence
$$\Omega^2 Y\simeq \Omega X\times \Omega W\simeq  Y\times W\times \Omega W.$$
Then $Y$ is a homotopy retract of $\Omega^2 Y$ and then of $\Omega^{2n} Y$ for any $n$. Notice that if $Y$ is homotopy retractible, there is nothing need to be proved. Hence, we may suppose $Y\not\simeq \ast$ and $\pi_i(Y)\not\cong 0$ for some $i$.
Then according to the above argument, we see $\pi_{i+2n}(Y)\not\cong 0$ for any $n$ which contradicts the homotopy finite assumption. Therefore $X$ can not be the homotopy retract of $\Omega Y$, and then $X\simeq Y$.

Now suppose by induction we have proved the proposition when the spaces involved can be decomposed into $n-1$ nontrivial irreducible $H$-spaces. Let $X$ be an $H$-space that can be decomposed into $n$ irreducible $H$-spaces. We can then write
$$X\simeq Z\times X^\prime,$$
such that $Z$ is an irreducible factor and ${\rm hodim}(X)={\rm hodim}(Z)$, where ${\rm hodim}(X)$ is the dimension of the top non-trivial homotopy group of $X$. Then (\ref{App1}) becomes
$$Z\times X^\prime\times \Omega Z\times \Omega X^\prime \simeq  Y\times \Omega Y.$$
Again by Lemma \ref{prime}, we see that $Z$ is a homotopy retract of $Y$. Hence $Y$ can be decomposed as
$$Y\simeq Z\times Y^\prime.$$
Then the above equivalence gives
$$X^\prime\times \Omega X^\prime \simeq Y^\prime\times \Omega Y^\prime$$
by either Theorem \ref{finite} or \ref{finitetype}.
Since $X^\prime$ can be decomposed into $n-1$ irreducible components, we have $X^\prime \simeq Y^\prime$ by induction and then $X\simeq Y$.
\end{proof}

\begin{acknowledgements}\label{ackref}
The authors would like to thank Haynes Miller for several valuable comments on our work. They also want to thank the referee most warmly for careful reading of our manuscript and numerous suggestions that have improved the exposition of this paper.

\end{acknowledgements}

\affiliationone{
   Ruizhi Huang\\
   Institute of Mathematics,
   Academy of\\
   Mathematics and Systems Science\\
   Chinese Academy of Sciences\\
   100190 Beijing
             \\
   China
   \email{huangrz@amss.ac.cn}}
\affiliationtwo{
   Jie Wu\\
   Department of Mathematics\\
   National University of Singapore\\
   119076 Singapore\\~
             \\
   Singapore
   \email{matwuj@nus.edu.sg}}
\end{document}